\documentclass[11pt]{article}

\usepackage{graphicx}
\usepackage{amsthm}
\usepackage{amsmath}
\usepackage{amsfonts}
\usepackage{amssymb}
\usepackage{amscd}
\usepackage[dvips]{hyperref}
\usepackage[usenames]{color}

\theoremstyle{plain}
\newtheorem{theorem}{Theorem}[section]
\newtheorem{lemma}[theorem]{Lemma}
\newtheorem{proposition}[theorem]{Proposition}

\theoremstyle{definition}
\newtheorem{definition}[theorem]{Definition}
\newtheorem{example}[theorem]{Example}

\newtheorem{remark}[theorem]{Remark}
\newtheorem{question}[theorem]{Question}

\def\surjection{\twoheadrightarrow}
\def\normal{\lhd}		

\def\Aut{\mathrm{Aut}}		

\def\N{\mathbb{N}}		
\def\Z{\mathbb{Z}}		
\def\Q{\mathbb{Q}}		
\def\R{\mathbb{R}}		
\def\C{\mathbb{C}}

\def\Sp{\mathrm{Sp}}		
\def\GL{\mathrm{GL}}		
\def\SL{\mathrm{SL}}		

\def\Mod{\mathrm{Mod}}		
\def\T{\mathcal{I}}		

\def\Teich{\mathcal{T}}		

\def\HH{\mathrm{H}}               

\def\Max{\mathrm{max}}

\newcommand\Dim{\ensuremath{dim}}

\title{Irreducible $\Sp$-representations and subgroup distortion in the 
mapping class group}
\author{Nathan Broaddus, Benson Farb, Andrew Putman\thanks{The first two authors are supported in part by the NSF.}}
\date{April 8, 2009}

\begin{document}
\maketitle

\begin{abstract} 
We prove that various subgroups of the mapping
class group $\Mod(\Sigma)$ of a surface $\Sigma$
are at least exponentially distorted.  Examples include the
Torelli group (answering a question of Hamenst\"adt), the
``point-pushing'' and surface braid subgroups,
and the Lagrangian subgroup.  Our techniques include
a method to compute lower bounds on distortion via
representation theory and an extension of Johnson theory to arbitrary
subgroups of $\HH_1(\Sigma;\Z)$.
\end{abstract}


\section{Introduction}
\label{section:introduction}

We begin with a basic motivating question.  
Manifolds $M$ are commonly presented as gluings or (possibly 
singular) fiberings of simpler manifolds, together with 
the data specifying the gluing or fibering.  In low 
dimensions, the gluing/fibering data commonly takes the form of
an element, or a finite collection of elements, of the {\em mapping class
group} $\Mod_g$, which is the group of homotopy classes of
orientation-preserving homeomorphisms of the closed, oriented, genus
$g$ surface $\Sigma_g$.  Examples include Heegaard decompositions
of $3$-manifolds, monodromies of surface bundles, and monodromies of 
Lefschetz fibrations. 

The topology of the resulting manifold $M$ can often be controlled by
requiring that the gluing data lie in various special subgroups
$K<\Mod_g$.  We then have a purely group-theoretic problem: 
determine whether or not a given element $f\in\Mod_g$, given as a
product of generators of $\Mod_g$ (say a generating set of Dehn twists),
lies in $K$.  The problem of finding such an algorithm is called solving
the {\em generalized word problem} for $K$ in $\Mod_g$.  The generalized
word problem is a classical problem in combinatorial group theory; it
was introduced and studied by Nielsen and Magnus.

A basic example is provided by
taking $K$ to be the {\em Torelli group} $\T_g<\Mod_g$, that is,
the kernel of the natural symplectic representation
$\pi:\Mod_g\to \Sp(2g,\Z)$ given by the action of $\Mod_g$ on
$\HH_1(\Sigma_g;\Z)$.  In this case, the solution to the
generalized word problem is easy: one simply computes the
induced linear map $f_\ast\in\Sp(2g,\Z)$ via matrix multiplication and
checks whether or not $f_\ast={\rm Id}$.  With this in mind, the more
refined and useful problem is to actually express a given element
of $\T_g$, written in generators for $\Mod_g$, in
terms of generators of the subgroup $\T_g$.  A standard quantitative
measure of the (in)efficiency of doing this is the notion of 
{\em distortion}, which we now explain.

\paragraph{Distortion in groups.}
Let $\Gamma$ be a finitely generated group
endowed with the word metric $\Vert \cdot \Vert_\Gamma$.  Any finitely
generated subgroup $K<\Gamma$ comes equipped with its own word metric
$\Vert \cdot \Vert_K$, and it is a basic problem in geometric group
theory to understand the geometry of the embedding $K \hookrightarrow
\Gamma$, that is, to compare the ``intrinsic metric'' $\Vert \cdot
\Vert_K$ with the ``extrinsic metric'' on $K$ given by the restriction
of $\Vert \cdot \Vert_\Gamma$ to $K$.  It is clear that there exist
constants $C, C'$ such that 
$$\Vert h \Vert_\Gamma \leq C \Vert h \Vert_K
+ C' \mbox{\ \ \ for all \ }h\in K.$$ 

However, there may be
``shortcuts'' in $\Gamma$ between elements of $K$.  This can be encoded
by a function $\delta:\N\to\N$, called the \emph{distortion} of $K$ in
$\Gamma$, which is defined to be the smallest function satisfying
$$\Vert h \Vert_K \leq \delta \left( \Vert h \Vert_\Gamma \right)
\mbox{\ \ \ for all \ }h \in K.$$ 

It is easy to see that $\delta$
changes by at most a linear factor if different generating sets are
chosen for $\Gamma$ or for $K$; thus the {\em growth type} of $\delta$
(e.g.\ polynomial of degree $d$, exponential, etc.)\ is independent of
these choices\footnote{Some people use terminology differing from ours
by a linear factor, so for example what we call ``linearly distorted''
they call ``undistorted''.}.  For an introduction to distortion and its
basic properties, see \cite{Gromov1993,Farb1994}.

We note here that for groups $\Gamma$ with solvable word problem, the
distortion of $K$ in $\Gamma$ gives a quantitative measure of the
efficiency of solving the generalized word problem for $K$ in $\Gamma$.
In particular, the generalized word problem for $K$ in $\Gamma$ is
solvable if and only if the distortion of $K$ in $\Gamma$ is recursive
(see \cite{Farb1994}). This problem can be unsolvable in simple
examples.  For instance, Mihailova (see, e.g.\ \cite{Farb1994}) found a
finitely generated subgroup $K$ in a product $F_m\times F_m$ of free
groups which has an unsolvable generalized word problem, and hence has
nonrecursive distortion.


\paragraph{Statement of results.}
In this paper we give bounds for the
distortions of various subgroups of $\Mod_g$.  Some
results in this direction are already known.  Stillwell
\cite{Stillwell} used Mihailova's example to find 
finitely generated subgroups of $\Mod_g$ with nonrecursive
distortion.  In every other case which has been previously investigated,
the distortion has turned out to be linear.  Examples of linearly
distorted subgroups of $\Mod_g$ include abelian subgroups \cite{FLM},
subgroups corresponding to mapping class groups of subsurfaces
\cite{MM2} (see \cite{Ham1} for another proof), and convex, cocompact subgroups \cite{FMo}.

These results led Hamenst\"adt to pose
the problem (see \cite[Problem~6]{Ham2}) of finding subgroups of 
$\Mod_g$ with (recursive) super-linear distortion.  In particular, 
she asked if the Torelli group 
$\T_g$ is linearly distorted in $\Mod_g$.  Recall that $\T_g$ is 
finitely generated for $g>2$ by a deep theorem of 
Johnson \cite[Theorem~2]{Johnson1983}.  It was conjectured in
\cite{Farb2006} (Problem 3.7 and the discussion following it) that
$\T_g$ has exponential distortion in $\Mod_g$.  Our first result
confirms the lower bound of this conjecture.  

To state our results in their full generality, let $\Sigma_{g,b}^p$ be
an oriented genus $g$ surface with $b$ boundary components and $p$
punctures (thought of as marked points), and let $\Mod_{g,b}^p$ be the
group of homotopy classes of homeomorphisms of $\Sigma_{g,b}^p$ which
fix the punctures and the boundary components pointwise (either $b$ or
$p$ will be omitted from our notation if they equal $0$).  When $b \in
\{0,1\}$, the {\em Torelli group} $\T_{g,b}$ is defined as 
the kernel of the action of
$\Mod_{g,b}$ on $\HH_1(\Sigma_{g,b};\Z)$.

\begin{theorem}[Distortion of the Torelli group]
\label{thm:torelli_distortion}
For $g\geq 3$, the distortion of $\T_g$ (resp. $\T_{g,1}$) in $\Mod_g$
(resp. $\Mod_{g,1}$) is at least exponential and at most doubly
exponential.
\end{theorem}

We conjecture that the upper bound in
Theorem~\ref{thm:torelli_distortion} is, like the lower bound,
exponential.  The upper bound is strongly related to the isoperimetric
and isodiametric inequalities in $\Sp(2g,\Z)$.  Thurston has conjectured
that for $n \geq 4$ the group $\SL(n,\Z)$ satisfies a quadratic isoperimetric
inequality.  The analogous conjecture for $\Sp(2g,\Z)$ (together with a theorem
of Papasoglu) would imply our conjectured upper bound.  See \S
\ref{section:torelli_distortion} below, especially 
Remark~\ref{remark:symplectic_isos}, for a discussion.

For mapping class groups of surfaces with boundary components or
punctures, one can construct finitely generated, normal subgroups by
``filling in the punctures and boundary components''; these are the
so-called {\em surface braid groups} (see \cite{Birman}).  
For example, the {\em point pushing subgroup} 
$\pi_1(\Sigma_g)\normal \Mod_g^1$ is
the kernel of the surjection $\Mod_g^1 \surjection \Mod_g$ induced by
``filling in the puncture''.  
 
\begin{theorem}[Distortion of surface braid groups]
\label{thm:surface_braids_distortion}
Let $K$ be the kernel of the surjection 
$$\Mod_{g,b+n}^{p+m} \surjection \Mod_{g,b}^{p},$$ 
where at least one of $n$ and $m$ is strictly greater than 0. Then 
$K$ is exactly exponentially distorted in $\Mod_{g,b+n}^{p+m}$ for $g \geq 2$.
\end{theorem}

\noindent
We prove Theorem~\ref{thm:surface_braids_distortion} 
in \S \ref{section:pointpushing_distortion}.  

In \S \ref{section:lagrangian_distortion} we construct ``relative
Johnson homomorphisms'', relative to arbitrary subgroups of
$\HH_1(\Sigma_g;\Z)$.  Applying these homomorphisms together with the general method for obtaining
lower bounds on distortion given in \S\ref{section:lowerbounds}, we
obtain Theorems~\ref{thm:lagrangian_distortion} and \ref{thm:pullback_distortion}
below as well as a general result, Theorem~\ref{thm:gentorellidistorted}, given in \S \ref{section:lagrangian_distortion}.

\begin{theorem}[Distortion of the Lagrangian subgroup]
\label{thm:lagrangian_distortion}
For $b \in \{0,1\}$, let $L < \Mod_{g,b}$ be the {\em Lagrangian
subgroup}, that is, the group of mapping classes which act trivially on
a fixed maximal isotropic subgroup of $\HH_1(\Sigma_{g,b};\Z)$.  Then
$L$ is at least exponentially distorted in $\Mod_{g,b}$ for $g \geq 4$.
\end{theorem}

\begin{remark}
The Lagrangian subgroup was first defined by Garoufalidis and Levine \cite{GarLevineFiniteType} and plays 
an important role in the theory of finite-type invariants of 3-manifolds.
\end{remark}

Masur-Minsky \cite{MM2} (see \cite{Ham1} for an alternate proof) proved
that for $h<g$ the natural inclusion $\Mod_{h,1}\hookrightarrow \Mod_g$
given by subsurface inclusion $\Sigma_{h,1}\hookrightarrow\Sigma_g$ is
linearly distorted (``undistorted'' in their terminology).  We also
remark that, as a consequence of \cite{LMR}, if $1<h<g$ then $\Sp(2h,\Z)$
is linearly distorted in $\Sp(2g,\Z)$.  One might therefore expect that 
the subgroup of $\Mod_g$ consisting of mapping classes
``homologically supported'' on $\Sigma_{h,1}$ have linear distortion.  
In contrast we have the following.

\begin{theorem}
\label{thm:pullback_distortion}
Suppose $g - h \geq 2$.  Let $\Sigma_{h,1}$ be an embedded subsurface of
$\Sigma_g$, and let $K$ be the pull-back to $\Mod_g$ of the
corresponding copy of $\Sp(2h,\Z)$ in $\Sp(2g,\Z)$.  Then $K$ is at
least exponentially distorted in $\Mod_g$.
\end{theorem}

We would like to know what happens when $g=h+1$, as our methods do not seem
to work in this case.

\paragraph{A first idea.}
The first key idea in this paper can be illustrated
by the following proof sketch of the lower bound in
Theorem~\ref{thm:torelli_distortion} in the case of $\T_g$ with $g \geq
3$.
\label{idea}
We begin with the standard exact sequence 
$$1 \longrightarrow \T_g \longrightarrow \Mod_g \longrightarrow \Sp(2g,\Z) \longrightarrow 1$$ 
coming from the action of $\Mod_g$ on
$H:=\HH_1(\Sigma_g;\Z) \cong \Z^{2g}$.  Let $U=\wedge^3H/H$.  Johnson
proved in \cite{Johnson1980} that there is a surjective homomorphism
$$\tau:\T_g \to U$$ 
which is equivariant with respect to the natural
actions of $\Mod_g$.  Here the $\Mod_g$-action on $U$ factors through
the standard $\Sp(2g,\Z)$ action, and the $\Mod_g$ action on $\T_g$ is 
the one induced by conjugation.  This equivariance is just the
formula
\begin{equation}
\label{eq:equiv}
\tau(fhf^{-1})=f_\ast \tau(h) \ \ \ \ \ \mbox{for all $f\in\Mod_g, h\in\T_g$,}
\end{equation}
\noindent
where $f_\ast$ denotes the induced action of $f\in\Mod_g$ on
$\HH_1(\Sigma_g;\Z)$.  

We can use $\tau$ to give lower bounds for word length in $\T_g$ as
follows.  Fix a finite generating set $S$ for $\T_g$, and let
$\Vert h\Vert_{\T_g}$ denote {\em word length} of $h$ with respect to
$S$, that is, the minimal number of elements of $S^{\pm 1}$ whose product equals
$h$.  Also, fix a norm $\Vert \cdot \Vert_U$ on $U$.
Since $S$ is finite, there exists some $C>0$ such that 
$\Vert \tau(s)\Vert_U\leq C$ for all $s\in S$.  Since $\tau$ is a homomorphism,
it is then clear that for every $h\in\T_g$ we have
\begin{equation}
\label{eq:compare}
\Vert h\Vert_{\T_g}\geq \frac{\displaystyle 1}{\displaystyle C}\Vert
\tau(h)\Vert_U.
\end{equation}

Now it is not hard to see that we can choose a mapping class $f \in \Mod_g$ with the property that
the linear transformation $f_{\ast} \in \GL(U)$ is {\em partially hyperbolic}, that is, there is some
eigenvalue $\lambda$ of $f_\ast$ with $|\lambda|>1$.  Using the partial hyperbolicity of $f$, we will 
find an $h \in \T_g$ such that $\Vert f_\ast^n(\tau(h))\Vert_U$ grows exponentially in $n$.
Then $\Vert f^n h f^{-n} \Vert_{\Mod_g}$ grows at most linearly with respect to $n$, but 
by (\ref{eq:equiv}) and (\ref{eq:compare}) we have
\begin{equation*}
\begin{split}
\Vert f^nhf^{-n}\Vert _{\T_g} &\geq \frac{\displaystyle 1}
{\displaystyle C}\Vert \tau(f^nhf^{-n})\Vert _U \\
&=\frac{\displaystyle 1}{\displaystyle C} \Vert f_\ast^n(\tau(h))\Vert_U, 
\end{split}
\end{equation*}
which grow exponentially.  See \S\ref{section:torelli_distortion} for details.

\paragraph{Some questions.}
This paper is a first attempt at a systematic study of
distortion in mapping class groups.  While the methods here apply to
many examples, there are also many examples to which they do not apply.
Of the many natural questions one might ask, we would like to point out
a particular one to which our methods do not seem to apply.  
Recall that the {\em handlebody group} is the subgroup of $\Mod_g$
consisting of elements which extend to a fixed handlebody.  Suzuki
\cite{Suzuki} proved that this group is finitely generated.

\begin{question}
What is the distortion of the handlebody subgroup in the mapping class group?
\end{question}

A natural question analogous to the direction of this paper is to study
and compute the distortions of orbits of subgroups $K<\Mod_g$ in the
Teichm\"{u}ller space $\Teich_g$ for $\Sigma_g$, say endowed with the
Teichm\"{u}ller metric; see \cite{FMo, KentLeininger} for related 
discussions.  Theorem~2.1 of \cite{FLM} (and the discussion following
it) give that any $\Mod_g$-orbit in $\Teich_g$ is exponentially
distorted.  Thus the problem of computing the distortion of $K$ in
$\Mod_g$ and of a $K$-orbit in $\Teich_g$ are {\it a priori} 
different.  We believe both questions are worth pursuing.

\paragraph{Acknowledgments.}
We are grateful to Danny Calegari for his help in
refining the picture for
Proposition~\ref{proposition:distortion_upper_bound}, to Jordan
Ellenberg for a helpful remark on the proof of Proposition~\ref{proposition:rep_criterion}, and
to Hanna Bennett for pointing out an error in a previous version of this paper.  We also thank the referee for several useful comments and corrections.


\section{Methods for bounding distortion in groups}

In this section, we give two general methods for bounding the distortion
of one group inside another, one yielding lower bounds and the other
upper bounds.  We will apply these methods throughout this paper to
subgroups of mapping class groups.  

\subsection{Behavior of distortion under inclusions}

Before we begin, we will need to know the following basic
property of distortion.

\begin{lemma}[Behavior of distortion under inclusions]
\label{lemma:stupid_distortion_lemma}
Let $K < \Gamma$ be finitely generated groups.  If $\Gamma < \Gamma'$ with $\Gamma'$ finitely generated,
then the distortion of $K$ in $\Gamma'$ is at least the distortion of $K$ in $\Gamma$.
\end{lemma}
\begin{proof}
Fix a finite generating set $S_\Gamma$ for $\Gamma$
and let $\Vert \cdot \Vert_\Gamma$ be the
associated word metric.  We can choose a finite generating
set $S_{\Gamma'}$ for $\Gamma'$ with $S_\Gamma \subset S_{\Gamma'}$.  Letting
$\Vert \cdot \Vert_{\Gamma'}$ be the
associated word metric, we have $\Vert k \Vert_\Gamma \geq \Vert k \Vert_{\Gamma'}$
for all $k \in K$, and the lemma follows.
\end{proof}

\subsection{Lower bounds via irreducible representations}
\label{section:lowerbounds}

If $K<G$, then to give a lower bound for the distortion of $K$ in $G$,
one must be able to give lower bounds on the word length in $K$.  This is
in general a very difficult problem, but for free abelian groups $V$ such bounds
can be easily obtained.  Moreover, the resulting linear algebra is a rich
source of examples of exponential growth.  We say that
an element of the automorphism group $\GL(V):=\Aut(V)$ is
{\em partially hyperbolic} if the corresponding linear
transformation of $V \otimes \C$ has some eigenvalue $\lambda$ with
$|\lambda|>1$.  We then have the following easy example of exponential
distortion.

\begin{example}
\label{example:mapping_torus}
Let $\Gamma$ be the semidirect product of $V=\Z^n$ with any subgroup 
of $\SL(n,\Z)$ which contains a partially hyperbolic matrix $A$.  Then $V$
has exponential distortion in $\Gamma$.  The upper bound is easy.  The 
exponential lower bound follows, as explained in the ``First idea'' on
page \pageref{idea} and given in detail in Proposition~\ref{proposition:lower_bound_distortion} below, from the exponential 
growth of the image of vectors in $V\otimes\R$ under iteration of $A$.  
\end{example}

\noindent
To generalize this example, we will map group/subgroup pairs to mapping 
tori of abelian groups like those in Example~\ref{example:mapping_torus},
using the abelian group as a sort of ``detector'' of exponential distortion.

\begin{proposition}[Criterion for exponential distortion]
\label{proposition:lower_bound_distortion}
Let $K$ be a finitely generated normal subgroup of a finitely
generated group $\Gamma$.  Suppose that $V$ is a free abelian group
equipped with a $\Gamma$-action $\rho:\Gamma \to \GL(V)$ and that
$\psi:K\to V$ is a surjective homomorphism which is
$\Gamma$-equivariant, where $\Gamma$ acts on $K$ by conjugation.  
If $\rho(\Gamma)$ contains a partially hyperbolic
matrix, then the distortion of $K$ in $\Gamma$ is at least
exponential.
\end{proposition}

\begin{remark}
When applying Proposition~\ref{proposition:lower_bound_distortion}, we will
find it useful to think of the exact sequence 
$$1 \longrightarrow K \longrightarrow \Gamma \longrightarrow \Gamma/K \longrightarrow 1$$
and to note that since $V$ is abelian and $\psi:K\to V$ is surjective we have
$$\rho(k) \cdot \psi(y) = \psi(kyk^{-1}) = \psi(k) + \psi(y) - \psi(k) = \psi(y) \ \ \ \ \mbox{for all $k,y \in K$}.$$
Thus the action of $K$ on $V$ is trivial.  Hence $\rho$ factors through a representation $\Gamma/K \to \GL(V)$. 
\end{remark}

\begin{remark}
By using Lemma~\ref{lemma:stupid_distortion_lemma}, Proposition~\ref{proposition:lower_bound_distortion} can be applied beyond the
context of normal subgroups.  See the proofs of Theorem~\ref{thm:lagrangian_distortion} and Theorem~\ref{thm:pullback_distortion} for
examples.
\end{remark}

\begin{proof}[Proof of Proposition~\ref{proposition:lower_bound_distortion}]
The hypothesis that $\psi$ is $\Gamma$-equivariant is precisely that
\begin{equation}
\label{eq:dist1}
\psi(xyx^{-1})=\rho(x) \cdot \psi(y) \ \ \ \ \mbox{for all $y \in K$ and $x \in \Gamma$}.
\end{equation}
By assumption, there exists some $x \in \Gamma$ such that
$$\rho(x) \otimes 1 : V \otimes \C \to V \otimes \C$$ 
has an eigenvalue $\lambda_1$ with $|\lambda_1|>1$.  Let $M = \rho(x) \otimes 1$ and let the 
distinct eigenvalues of $M$ be
$\lambda_1,\ldots,\lambda_k \in \C$.  For $1 \leq i \leq k$, let $E_k$ be the 
{\em generalized eigenspace} for $\lambda_i$, i.e.\ the
kernel of the map $(M - \lambda_i \mathbb{I})^{\Dim(V)}$, where $\mathbb{I}$ 
is the identity.  We have a direct sum decomposition $V \otimes \C = E_1 \oplus \cdots \oplus E_k$
which is invariant under $M$.  Endow $V \otimes \C$ with an inner product such that the $E_i$ are orthogonal
to each other and let $\Vert \cdot \Vert$ be the associated norm.  It is well-known (see, e.g., \cite[\S 2.2]{KatokHasselblatt})
that there exists some $C > 0$ such that if $v \in E_1$, then 
$\Vert M^n \cdot v \Vert \geq C |\lambda_1|^n \Vert v \Vert$ for all $n \geq 1$.

Let $S$ be a finite generating set for $K$ and let $S' \supseteq S$
be a finite generating set for $\Gamma$.  
Since $S$ is finite, there exists some $D \geq 0$ such that 
\begin{equation}
\label{eq:dist2}
\Vert\psi(s) \otimes 1\Vert \leq D \quad \quad \text{for each $s\in S$}.
\end{equation}  
Since $\psi$ is surjective, its image in $V \otimes \C$ must span $V \otimes \C$.  In particular, its image cannot
lie in $E_2 \oplus \cdots \oplus E_k$, so we can find some $y \in K$ such that 
$\psi(y) \otimes 1 \notin E_2 \oplus \cdots \oplus E_k$.  Let $y_1$ be the orthogonal projection of
$\psi(y) \otimes 1$ to $E_1$.  Using the fact that the $E_i$ are orthogonal and invariant under
$M$, we deduce that for all $n \geq 1$ we have
\begin{equation}
\label{eq:dist3}
\begin{split}
\Vert \psi(x^nyx^{-n}) \otimes 1\Vert &= \Vert M^n \cdot (\psi(y) \otimes 1) \Vert \\
&\geq \Vert M^n \cdot y_1 \Vert \\
&\geq C |\lambda_1|^n \Vert y_1 \Vert.
\end{split}
\end{equation}

Now (\ref{eq:dist2}) together with (\ref{eq:dist3}) implies that the 
word length $\Vert x^nyx^{-n}\Vert_K$ with respect to the generating set $S$ is at least 
$\frac{1}{D}\cdot C |\lambda_1|^n \Vert y_1 \Vert$, which grows
exponentially in $n$ since $|\lambda_1| > 1$ and $y_1 \neq 0$.  On
the other hand, since $x$ and $y$ have fixed word length with respect to
the generating set $S'$ for $\Gamma$, we see that the word length in $\Gamma$ of $x^nyx^{-n}$
grows at most linearly in $n$.  This completes the proof of the proposition.
\end{proof}

\noindent
The following (purely linear-algebra) 
proposition gives a useful way in practice to prove that 
an action contains some partially hyperbolic matrix.

\begin{proposition}[Representation-theoretic criterion for hyperbolicity]
\label{proposition:rep_criterion}
Let $\Gamma$ be a group and let $V$ be a free abelian group
equipped with a $\Gamma$-action.  Suppose that $V$ contains
a nontrivial $\Gamma$-submodule $W$ satisfying the following two properties:
\begin{enumerate}
\item $W$ is {\em irreducible} in the sense that if $W' < W$ is a
nontrivial $\Gamma$-submodule, then
$W' \otimes \R = W \otimes \R$.
\item The image of the map $\nu : \Gamma \to \GL(W)$ is infinite.
\end{enumerate}
Then $\nu(\Gamma)$ contains a partially hyperbolic matrix.
\end{proposition}

\begin{remark}
In all the examples we consider in this paper, one could prove that
the relevant representations contain partially hyperbolic matrices
by hand; however, we believe that Proposition \ref{proposition:rep_criterion}
gives a conceptual reason for the ubiquity of representations with this property.
\end{remark}

\begin{proof}[{Proof of Proposition \ref{proposition:rep_criterion}}]
Choosing a basis, we can identify $\GL(W)$ with $\GL(m,\Z)$ for some
$m\geq 1$.  If a matrix $A\in\GL(W)$ is not partially hyperbolic, then
every eigenvalue $\lambda$ of $A$ satisfies $|\lambda|=1$.
Since $A$ has integer entries, each of its eigenvalues is an algebraic
integer.  An old theorem of Kronecker (see, e.g., \cite{Gr}) 
states that if $\lambda$ is any algebraic integer 
with $|\lambda|=1$ and with $|\lambda'|=1$ for every Galois conjugate
$\lambda'$ of $\lambda$, then $\lambda$ is a root of unity.  

Let $\GL(n,\Z)[L]$ denote the {\em level $L$ congruence subgroup} of
$\GL(n,\Z)$, which consists of those $A\in\GL(n,\Z)$ for which, when the
entries of $A$ are taken $\text{mod $L$}$, the result is the identity
matrix.  Theorem 5.61 of \cite{Mo} states that there exists some $L>1$
such that no eigenvalue of any element of $\GL(n,\Z)[L]$ is a nontrivial root of
unity.  Pulling back $\GL(n,\Z)[L]$ via $\nu$ then gives us
a finite index subgroup $\Gamma'$ of $\Gamma$ with the property that no
eigenvalue of any element of $\nu(\Gamma')$ is a nontrivial root of
unity.  Since we are assuming (assumption (2)) that $\nu(\Gamma)$ is 
infinite, we know $\nu(\Gamma')$ is nontrivial (indeed infinite).

Suppose $\nu(\Gamma')$ contains no partially hyperbolic matrix.  Then by
the above two paragraphs, every $A\in\nu(\Gamma')$ must be
{\em unipotent}, that is, $A$ has $1$ as its only eigenvalue.  But any
subgroup of $\GL(n,\R)$ consisting of unipotent matrices must be
nilpotent, and indeed conjugate into the upper triangular group with
$1$'s on the diagonal (see, e.g., Corollary 17.5 of \cite{Hu}).  Setting
$$U=\{\text{$w\in W$ $|$ $\Gamma'\cdot w=w$}\},$$
this implies that $U \neq 0$.
Since $\nu(\Gamma)$ is infinite (assumption (2)), we have $U \neq W$, for otherwise
$\nu(\Gamma')$ would be trivial and so $\nu(\Gamma)$ would be finite.  Also, since
$\Gamma'$ is a normal subgroup of $\Gamma$, we have that $\Gamma \cdot U \subset U$; indeed,
for $h \in \Gamma$ and $u \in U$, we have for all $g \in \Gamma'$ that
$$g \cdot (h \cdot u) = h \cdot ((h^{-1} g h) \cdot u) = h \cdot u,$$
so $h \cdot u \in U$.  We conclude that $\nu(\Gamma)$ is reducible, contradicting assumption (1).  Thus
$\nu(\Gamma') \subset \nu(\Gamma)$ must contain some partially hyperbolic matrix, and we are done.
\end{proof}

To apply Proposition~\ref{proposition:rep_criterion}, we will need a
method for establishing the irreducibility of the action of the discrete
group $\Gamma$.  In the cases which arise in this paper, this action
will factor through various arithmetic subgroups (such as $\Sp(2g,\Z)$)
of semisimple Lie groups, and hence the following theorem of Borel can be
applied.

\begin{theorem}[{\cite[Proposition~3.2]{Borel}}]
\label{thm:irreducibles_agree}
Let $G$ be a connected, semisimple, real algebraic group
which is defined over $\Q$ and which has no compact factors (for example,
$G=\Sp(2g,\R)$ for $g \geq 1$ or $G=\SL(n,\R)$ for $n \geq 2$).  
Let $V$ be a finite dimensional vector space over $\R$ which
is an irreducible $G$-module.  Then $V$ is also an irreducible
$G_\Z$-module, where $G_\Z$ is the group of integer points of $G$.
\end{theorem}

\begin{remark}
The conditions in \cite[Proposition~3.2]{Borel} do not mention integer points, but instead
refer to subgroups satisfying a certain ``property (S)''.  However, it is easy to see that
lattices satisfy this property, and a well-known theorem of Borel-Harish-Chandra says that
arithmetic subgroups are lattices.  We also remark that \cite[Proposition~3.2]{Borel} is
one of the steps in the proof of an early version of the Borel Density Theorem, but later
proofs do not make use of this result, and Theorem~\ref{thm:irreducibles_agree} can 
be easily deduced from the Borel Density Theorem.
\end{remark}


\subsection{Isoperimetric inequalities and upper bounds on distortion}
\label{section:upper_bounds_on_distortion}

The goal of this subsection is to prove that an upper bound for
the distortion of a normal subgroup $K\normal \Gamma$ can be obtained
from isoperimetric and isodiametric functions on the quotient group 
$\Gamma/K$.  We first recall some definitions.  While the notion of isoperimetric
function and isodiametric function for a group are usually defined as
independent quantities (see, e.g., \cite{Gersten1993}), we will be
interested in their simultaneous realization, as discussed in 
\cite{GR} and in \cite[Definition 2.2.2]{RileyNotes}.  We refer the
reader to these references for background on this topic.

\begin{definition}[Simultaneous (isoperimetric,isodiametric) pairs] 
\label{definition:isoperimetric_inequality}
Let $\Gamma \cong \langle S | R \rangle$ be a finitely-presented group.  Denote
the free group on $S$ by $F(S)$.  Also, denote the word length of an element
$w$ in $F(S)$ by $\ell(w)$.  We say that a pair of
functions $(\mu_P,\mu_D)$ is a {\em simultaneous (isoperimetric,isodiametric) pair 
for $\Gamma$} if for all $w\in F(S)$ representing the identity in $\Gamma$, we can write 
$w$ as a word
$$w = \prod_{i=1}^{N} x_i r_i x_i^{-1},\ \ x_i\in F(S), r_i\in R$$
such that
$$N\leq \mu_P(\ell(w))\ \ \mbox{and}\ \ \ell(x_i)\leq \mu_D(\ell(w))\ \
\mbox{for each $1\leq i\leq N$}.$$
\end{definition}

\begin{remark}
Though the precise functions in an (isoperimetric,isodiametric) pair for
a group depend on the choice of finite generating set, changing the
generators preserves the equivalence classes of these functions up to
linear substitution.  Thus as with distortion we will consider
(isoperimetric,isodiametric) pairs only up to their well-defined
equivalence classes.
\end{remark}

\noindent
The main result of this subsection is the following, which 
is a small generalization of a theorem of Arzhantseva and Osin \cite[Lemma 3.6]{ArzhantsevaOsin}.

\begin{proposition}[Upper bound on distortion]
\label{proposition:distortion_upper_bound}
Let $K \normal \Gamma$ be a finitely generated, normal subgroup of a
finitely generated group $\Gamma$.  Suppose that $\Gamma / K$ is finitely presented and that $(\mu_P,\mu_D)$ is a simultaneous
(isoperimetric,isodiametric) pair for $\Gamma / K$.  Then there exists some $C>0$
such that the distortion of $K$ in $\Gamma$ is at most $\mu_P C^{\mu_{D}}$.
\end{proposition}
\begin{proof}
Let $S_\Gamma$ and $S_K$ be finite generating sets for $\Gamma$ and $K$ respectively, and let $F(S_\Gamma)$ be the free group on $S_\Gamma$.   Let $\Vert\cdot\Vert_{\Gamma}: \Gamma \to \N$ be the $S_\Gamma$ word metric and $\Vert \cdot \Vert_K : K \to \N$ be the $S_K$ word metric. Let $\ell: F(S_\Gamma) \to \N$ give the word length in $F(S_\Gamma)$ and $\pi : F(S_\Gamma) \to \Gamma$
be the natural projection.  Finally, choose $R \subset F(S_\Gamma)$ such that $\Gamma / K \cong \langle S_{\Gamma} | R \rangle$ is a finite presentation.

We begin by claiming that there is some $C>0$ such that if $r \in R$ and $x \in F(S_\Gamma)$, then
$$\Vert \pi(x r x^{-1}) \Vert_K < C^{\ell (x)}.$$
Indeed, an easy induction shows that
if
$$C_1 = \Max\{\text{$\Vert\pi(r)\Vert_K$ $|$ $r \in R$}\},$$
and 
$$C_2 = \Max\{\text{$\Vert\pi(s_\Gamma s_K s_\Gamma^{-1})\Vert_K$ $|$ $s_\Gamma \in S_\Gamma$ and $s_K \in S_K$}\},$$
then $C = C_1C_2$ satisfies the claim.

Now given $k \in K$ choose an efficient word $w \in F(S_\Gamma)$ such that $\pi(w) = k$ and $\ell (w) = \Vert k \Vert_\Gamma$.
By Definition~\ref{definition:isoperimetric_inequality}, we can write
$$w = \prod_{i=1}^{N} x_i r_i x_i^{-1},\ \ x_i\in F(S_{\Gamma}),\ \ r_i\in R$$
with
$$N\leq \mu_P(\ell ( w ) )\ \ \mbox{and}\ \ \ell (x_i) \leq \mu_D(\ell ( w ))\ \
\mbox{for each $1\leq i\leq N$}.$$
We then calculate:
\begin{align*}
\Vert k \Vert_K &=    \Vert\pi (w) \Vert_K \\
                &\leq \sum_{i=1}^N \Vert\pi (w_i r_i w_i^{-1})\Vert_K \\
                &\leq \sum_{i=1}^N C^{\mu_D(\ell (w) )} \\
                &\leq \mu_P(\ell (w)) C^{\mu_D(\ell (w))} \\
                &=    \mu_P(\Vert k \Vert_\Gamma) C^{\mu_D(\Vert k \Vert_\Gamma)},
\end{align*}
as desired.
\end{proof}


\section{Distortion of the Torelli group}
\label{section:torelli_distortion}

In this section, we apply Propositions~\ref{proposition:lower_bound_distortion} and
\ref{proposition:distortion_upper_bound} to give lower and upper bounds on the
distortion of $\T_{g,b}$ in $\Mod_{g,b}$ for $b \in \{0,1\}$ and $g \geq 3$.

\begin{proof}[Proof of Theorem~\ref{thm:torelli_distortion}]
Consider the standard exact sequence
\begin{equation}
\label{torseq}
1 \longrightarrow \T_{g,b} \longrightarrow \Mod_{g,b} \longrightarrow \Sp(2g,\Z) \longrightarrow 1. 
\end{equation}

We begin with the lower bound.  Set $H = \HH_1(\Sigma_{g,b};\Z)$.  In \cite{Johnson1980}, Johnson constructed the
{\em Johnson homomorphisms}, which are for $g \geq 3$ surjective 
$\Mod_{g,b}$-equivariant homomorphisms
$$\tau : \T_{g,1} \to \wedge^3 H$$ 
and 
$$\tau : \T_g \to (\wedge^3 H) / H$$
Here $H$ is embedded in $\wedge^3 H$ as
$H \wedge \omega$, where $\omega = a_1 \wedge b_1 + \cdots + a_g \wedge
b_g$ for any symplectic basis $\{a_1,b_1,\ldots,a_g,b_g\}$ of $H$.
The action of $\Mod_{g,b}$ on $(\wedge^3 H) / H$ or $\wedge^3 H$ factors surjectively through the infinite group $\Sp(2g,\Z)$.

Now $H \otimes \R$ is an irreducible representation of $\Sp(2g,\R)$.  By Theorem~\ref{thm:irreducibles_agree} it is also an irreducible $\Sp(2g,\Z)$-representation.  We can apply Proposition~\ref{proposition:rep_criterion} to the inclusion $H \subset \wedge^3 H$ to show that some element of $\Mod_{g,1}$ acts partially hyperbolically on $\wedge^3 H$.
Similarly, $((\wedge^3 H) / H) \otimes \R$ is an irreducible representation of $\Sp(2g,\R)$. By Theorem~\ref{thm:irreducibles_agree} it is also an irreducible $\Sp(2g,\Z)$-representation. We can apply Proposition~\ref{proposition:rep_criterion} with $W=V=(\wedge^3 H) / H$ to show that some element of $\Mod_{g,0}$ acts partially hyperbolically on $(\wedge^3 H) / H$.

Now apply Proposition~\ref{proposition:lower_bound_distortion} to $\T_{g,b} \normal \Mod_{g,b}$ with the homomorphism $\psi$ equal to the relevant
Johnson homomorphism to conclude that $\T_{g,b}$ is at least exponentially
distorted in $\Mod_{g,b}$ for $g \geq 3$ and $b \in \{0,1\}$, as desired.

We now establish the upper bound.  As explained in \S 5 of
\cite{Leuzinger} (see in particular Corollary~5.4), the nonpositively
curved symmetric space $X$ for $\Sp(2g,\R)$ admits an
$\Sp(2g,\Z)$-equivariant retraction $r:X\to\Omega$ onto a submanifold with boundary
$\Omega$ on which $\Sp(2g,\Z)$ acts cocompactly and properly by
isometries (in the path metric).  Now $X$ is a nonpositively curved
Riemannian manifold, and so it has a simultaneous
(isoperimetric,isodiametric) function which is (quadratic, linear).  The
retraction $r$ distorts lengths and hence volumes by at most an
exponential factor.  It follows that $\Omega$, hence $\Sp(2g,\Z)$, has
at worst a simultaneous (exponential,exponential)
(isoperimetric,isodiametric) pair.  While \cite{Leuzinger} explains this
only for isoperimetric functions, the argument for the
(isoperimetric,isodiametric) pair follows exactly his argument.

Proposition~\ref{proposition:distortion_upper_bound} applied to exact sequence (\ref{torseq})
therefore implies that for $g \geq 3$ and $b \in \{0,1\}$, the group
$\T_{g,b}$ is at most doubly exponentially distorted in $\Mod_{g,b}$, as
desired.
\end{proof}

 \begin{remark}
\label{remark:symplectic_isos}
The isoperimetric/isodiametric bounds for $\Sp(2g,\Z)$ used in the proof
of Theorem~\ref{thm:torelli_distortion} are probably not sharp.  In
fact, Thurston has conjectured that $\SL(n,\Z)$ satisfies a quadratic
isoperimetric inequality for $n \geq 4$, and one would expect the same to
hold for $\Sp(2g,\Z)$ for $g \geq 3$.  If $\Sp(2g,\Z)$ satisfied a
quadratic isoperimetric inequality, then a theorem of
Papasoglu \cite{Papasoglu} (see \cite[Theorem~2]{GR} for a generalization
and alternate proof) would imply that it has a (quadratic,linear)
(isoperimetric,isodiametric) pair.  This would imply that our exponential lower bound on the distortion 
of $\T_g$ in $\Mod_g$ (and of $\T_{g,1}$ in $\Mod_{g,1}$) is sharp.  

In fact one could get by with much less than Thurston's
conjecture. Sharpness of the exponential lower bound on the distortion
of the Torelli group would follow if $\Sp(2g,\Z)$ satisfied an (exponential,linear) (isoperimetric,isodiametric) pair.
\end{remark}


\section{Distortion of surface braid groups}
\label{section:pointpushing_distortion}

\begin{proof}[Proof of Theorem~\ref{thm:surface_braids_distortion}]
Consider the exact sequence
\begin{equation}
\label{braidseq}
1 \longrightarrow K \longrightarrow \Mod_{g,b+n}^{p+m} \longrightarrow \Mod_{g,b}^p \longrightarrow 1.
\end{equation}

We first prove that $K$ is at least exponentially
distorted in $\Mod_{g,b+n}^{p+m}$.  Setting $H = \HH_1(\Sigma_g;\Z)$, we
will construct a $\Mod_{g,b+n}^{p+m}$-equivariant surjection $\psi : K \to H^{n+m}$.  Here
the $\Mod_{g,b+n}^{p+m}$-action on $H^{n+m}$ factors through the natural projection
$$\Mod_{g,b+n}^{p+m} \longrightarrow \Mod_g \longrightarrow \Sp(2g,\Z).$$
Since $H^{n+m}$ contains the nontrivial irreducible (by Theorem~\ref{thm:irreducibles_agree}) 
$\Sp(2g,\Z)$-module $H$, we can apply Propositions~\ref{proposition:rep_criterion} and \ref{proposition:lower_bound_distortion} to $K \normal \Mod_{g,b+n}^{p+m}$ 
with the homomorphism $\psi$ to conclude that $K$ is at least exponentially distorted in $\Mod_{g,b+n}^{p+m}$, as desired.

The construction of $\psi$ goes as follows.  Let $K'$ be the kernel of the surjection $\Mod_{g,b}^{p+m+n} \to \Mod_{g,b}^p$.
The map $\Mod_{g,b+n}^{p+m} \to \Mod_{g,b}^{p+m+n}$ induces a surjection $K \to K'$.  A basic result of
Birman \cite{Birman} shows that $K'$ is the fundamental group of the {\em configuration space of $n+m$ points
on $\Sigma_{g+b}^p$}, that is, that $K' \cong \pi_1((\Sigma_{g,b}^p)^{n+m} \setminus \Delta)$ where
$$\Delta = \{\text{$(x_1,\ldots,x_{n+m}) \in (\Sigma_{g,b}^p)^{n+m}$ $|$ $x_i = x_j$ for some $i \neq j$}\}.$$
The map $\psi$ is then the composition
$$K \longrightarrow K' = \pi_1 \big((\Sigma_{g,b}^p)^{n+m} \setminus \Delta \big)  \longrightarrow \pi_1 \big ((\Sigma_{g,b}^p)^{n+m} \big) \longrightarrow \pi_1 \big((\Sigma_g)^{n+m} \big) \longrightarrow H^{n+m}.$$

We now prove that $K$ is at most exponentially distorted in $\Mod_{g,b+n}^{p+m}$.
In \cite{Mosher}, Mosher proved that $\Mod_{g,b}^p$ is
automatic.  This implies \cite{ECHLPT} that $\Mod_{g,b}^p$ satisfies a simultaneous (quadratic,linear)
(isoperimetric,isodiametric) pair.  We can thus deduce the desired upper bound by applying
Proposition~\ref{proposition:distortion_upper_bound} to exact sequence (\ref{braidseq}).
\end{proof}


\section{Relative Johnson homomorphisms and the distortion of
homologically specified subgroups}
\label{section:other_subgroups}

In this section, we generalize the Johnson homomorphism to other
subgroups of $\Mod_{g,b}$ defined by a variety of homological
conditions.  We then apply these homomorphisms to give lower bounds on
distortion.


\subsection{Relative Johnson homomorphisms}
\label{section:generalized_johnson}

Fix $g \geq 3$ and $b \in \{0,1\}$, and set $H = \HH_1(\Sigma_{g,b};\Z)$.  We will consider
the following subgroups of the mapping class group $\Mod_{g,b}$.

\begin{definition}
Let $W$ be a subgroup of $H$.  Define
$$\Mod^W_{g,b} := \{\text{$f \in \Mod_{g,b}$ $|$ $f_{\ast}(W) \subset W$}\}.$$
Observe that $\Mod^W_{g,b}$ acts on $H/W$.  The kernel of this action will be
denoted by $\T^W_{g,b}$.
\end{definition}

\begin{example}
The classical Torelli group $\T_{g,b}$ corresponds to $\T^W_{g,b}$ with $W=0$.
\end{example}

\begin{example}
Fix $m \in \N$, and let $W$ be the kernel of the map $H \to \HH_1(\Sigma_{g,b};\Z / m \Z)$.  
Then $\Mod^W_{g,b} = \Mod_{g,b}$ and $\T^W_{g,b}$ is the {\em level $m$ subgroup of $\Mod_{g,b}$}, that is, the kernel
of natural map $\Mod_{g,b} \to \Sp(2g,\Z / m \Z)$.
\end{example}

\begin{example}
Let $\Sigma' \subset \Sigma_{g,b}$ be a subsurface, and let $W$ be the image
of $\HH_1(\Sigma';\Z)$ in $H$.  Then $\T^W_{g,b}$ is generated by $\T_{g,b}$
together with the set mapping classes supported on $\Sigma'$.  One may think
of this as the subgroup of mapping classes which are ``homologically
supported'' on $\Sigma'$.
\end{example}

\begin{example}
Let $W \subset \HH_1(\Sigma_{g,b};\Z)$ be a Lagrangian, that is, a maximal isotropic subgroup.  Then
$\T^W_{g,b}$ is the subgroup of mapping classes which preserve $W$ and act trivially on $H/W$.  It is easy
to see that $\T^W_{g,b}$ must also act trivially on $W$.  In other words, these are the Lagrangian subgroups
of Theorem \ref{thm:lagrangian_distortion}.
\end{example}

\noindent
We now wish to generalize the Johnson homomorphisms to the groups $\T^W_{g,b}$.  We begin by
discussing the appropriate target for these Johnson homomorphisms.  We will need the following well-known
lemma.

\begin{lemma}[{\cite[Theorem V.6.4]{BrownCohomology}}]
\label{lemma:johnsonimage}
If $A$ is an abelian group, then there is an injection $i:\wedge^3 A \to \HH_3(A;\Z)$.  If
$A$ is torsion-free, then $i$ is an isomorphism.
\end{lemma}

Now, using Lemma~\ref{lemma:johnsonimage}, the classical Johnson homomorphism on a
surface with boundary is of the form
$$\tau : \T_{g,1} \longrightarrow \wedge^3 H \cong \HH_3(H;\Z).$$
The relative Johnson homomorphism on a surface with boundary will be of the form
$$\tau^W : \T^W_{g,1} \longrightarrow \HH_3(H/W;\Z).$$
On a closed surface, the classical Johnson homomorphism is of the form
$$\tau : \T_g \longrightarrow (\wedge^3 H)/H \cong \HH_3(H;\Z) / H.$$
Here $H$ is embedded in $\wedge^3 H$ as
$H \wedge \omega$, where $\omega = a_1 \wedge b_1 + \cdots + a_g \wedge
b_g$ for any symplectic basis $\{a_1,b_1,\ldots,a_g,b_g\}$ of $H$.  By
Lemma~\ref{lemma:johnsonimage}, there is a copy of $\wedge^3 (H/W)$ in
$\HH_3(H/W;\Z)$.  For any subgroup $W$ of $H$, the aforementioned embedding of $H$ into $\wedge^3 H$
therefore projects to a homomorphism
$$H \hookrightarrow \wedge^3 H \to \wedge^3 (H/W) \hookrightarrow \HH_3(H/W;\Z)$$
whose image we will denote by $H^W$.  The relative Johnson homomorphism on
a closed surface will be of the form
$$\tau^W : \T^W_{g} \longrightarrow \HH_3(H/W;\Z) / H^W.$$

\begin{remark}
\label{remark:h_image}
The map $H \to H^W$ need not be injective.  Indeed, if $W$ is a maximal isotropic subspace of $H$, then
the image of $\omega$ in $\wedge^2 (H/W)$ is $0$.  Hence in this case $H^W=0$.
\end{remark}

\noindent
For simplicity, we will define $H^W = 0$ if $b = 1$.  We can now
state our theorem.
\begin{theorem}[Relative Johnson homomorphisms]
\label{theorem:relativejohnson}
For $g \geq 3$ and $b \in \{0,1\}$, there exist homomorphisms
$$\tau^W : \T^W_{g,b} \longrightarrow \HH_3(H/W;\Z) / H^W,$$
which we will call the {\em relative Johnson homomorphisms}, satisfying the following properties:
\begin{enumerate}
\item $\tau^0$ is the classical Johnson homomorphism.
\item If $W_1 \subset W_2$, then $\tau^{W_2}|_{\T^{W_1}_{g,b}}$ equals $\tau^{W_1}$ followed by
the natural map
$$\HH_3(H/W_1;\Z) / H^{W_1} \to \HH_3(H/W_2;\Z) / H^{W_2}.$$
\item For $f \in \Mod^W_{g,b}$ and $h \in \T_{g,b}$, we have $\tau(f h f^{-1}) = f_{\ast} \circ \tau(h)$, where
$f_{\ast}$ is the induced map on $\HH_3(H/W;\Z) / H^W$.
\end{enumerate}
\end{theorem}
\begin{proof}
We begin with the case $b = 1$ (where $H^W = 0$).  
As described for example in \cite{JohnsonSurvey}, the classical Johnson 
homomorphism $\tau$ can be defined in various ways.  
We imitate the construction
based on mapping tori.  For a detailed discussion of this definition of
the classical Johnson homomorphism and a proof of its equivalence to the
more standard definition, see \S 3-4 of \cite{HainTorelli}.  For $h
\in \T^W_{g,1}$, choose a homeomorphism $h'$ of $\Sigma_{g,1}$ representing $h$ 
and construct a homeomorphism $\overline{h}$ of $\Sigma_g$ by gluing a
disc $D$ to the boundary component of $\Sigma_{g,1}$ to get a copy of $\Sigma_g$ 
and defining $\overline{h}$ to equal $1$ on $D$ and $h'$ on $\Sigma_{g,1} = \Sigma_g \setminus D$.
Next, let $M_{\overline{h}}$ be the {\em mapping torus} of $\overline{h}$, that is,
the quotient of $\Sigma_{g} \times [0,1]$ by the equivalence relation
$(x,1) \sim (\overline{h}(x),0)$.  Fixing a basepoint $\upsilon$ on $D \times 0 \subset \Sigma_g \times 0$,
we get a canonical loop $\ell \in \pi_1(M_{\overline{h}},\upsilon)$, via $\ell(t) = \upsilon \times t$.  Observe that
$$\HH_1(M_{\overline{h}};\Z) / \langle \ell \rangle \cong H / \langle \{\text{$\alpha -
h_\ast \alpha$ $|$ $\alpha \in H$}\} \rangle.$$ 
Since $h \in \T^W_{g,1}$, this
has a natural projection to $H/W$.  We conclude that there is a natural
map 
$$\pi_1(M_{\overline{h}},\upsilon) \longrightarrow H/W,$$ 
and since $M_{\overline{h}}$ is an
Eilenberg-MacLane space, we have an induced map
$$\phi : M_{\overline{h}} \longrightarrow K(H/W,1).$$ 
We define 
$$\tau^W(h) = \phi_{\ast} ([M_{\overline{h}}]) \in \HH_3(H/W;\Z).$$
It is easy to see that this definition is independent of the choices involved
in its construction, and $\tau^W$ clearly satisfies Condition 2 of
the theorem.  Moreover, if $W = 0$,
then $\tau^W$ reduces to Johnson's definition of the classical Johnson
homomorphism via mapping tori, so Condition 1 follows.  The proof (see
\cite{JohnsonSurvey}) that
the classical Johnson homomorphism is a homomorphism which satisfies
Condition 3 generalizes verbatim to our situation, so we are done.

We now deal with the case $b = 0$.  We have an exact sequence
$$1 \longrightarrow \pi_1(T^1\Sigma_g) \longrightarrow \T^W_{g,1} \longrightarrow \T^W_g \longrightarrow 1,$$
where $T^1 \Sigma_g$ is the unit tangent bundle of $\Sigma_g$; see \cite{Johnson1983}.
The classical Johnson homomorphism restricted to $\pi_1(T^1\Sigma_g)$ lands in $H \subset \wedge^3 H$.  It follows
that $\tau^W$ restricted to $\pi_1(T^1\Sigma_g)$ lands in $H^W \subset \wedge^3 H/W$, and thus we have
an induced map
$$\T^W_g \longrightarrow (\HH_3(H/W;\Z))/ H^W.$$
These maps clearly satisfy the conditions of the theorem.
\end{proof}

\begin{remark}
In the special case of the Lagrangian subgroup, Levine
\cite{LevineFiniteType} has given a different construction of $\tau^W$.
His construction imitates Johnson's original definition of $\tau$
\cite[First Definition]{JohnsonSurvey}, while our construction is
inspired by Johnson's definition in terms of mapping tori \cite[Second
Definition]{JohnsonSurvey}.
\end{remark}

\begin{remark}
Another approach sufficient for the proofs of Theorems~\ref{thm:lagrangian_distortion} and \ref{thm:pullback_distortion} would be to use Morita's \cite{MoritaExtension} 
extension of the classical Johnson homomorphism to a crossed
homomorphism $\Mod_{g,1} \to \frac{1}{2} \wedge^3 H$.
\end{remark}


\subsection{Applications to subgroup distortion}
\label{section:lagrangian_distortion}

We now apply the relative Johnson homomorphisms $\tau^W$ together with 
Propositions~\ref{proposition:rep_criterion} and 
\ref{proposition:lower_bound_distortion} to give lower bounds on the 
distortions of some of the groups $\T_{g,b}^W$ inside $\Mod_{g,b}$.  
For these distortions to make sense, however, 
we must first prove the following.

\begin{proposition}
\label{proposition:relativetorellifinite}
For $g \geq 3$ and $b \in \{ 0,1 \}$, let $W$ be any subgroup 
of $\HH_1(\Sigma_{g,b};\Z)$.  Then
the groups $\T^W_{g,b}$ and $\Mod^W_{g,b}$ are finitely generated.
\end{proposition}

\begin{proof}
Let $\Gamma$ be either $\T^W_{g,b}$ or $\Mod^W_{g,n}$.  We have an exact
sequence $$1 \longrightarrow \T_{g,b} \longrightarrow \Gamma
\longrightarrow B \longrightarrow 1,$$ where $B$ is the image of $\Gamma
< \Mod_{g,b}$ in $\Sp(2g,\Z)$.  As has already been mentioned, Johnson
\cite{Johnson1983} showed that $\T_{g,b}$ is finitely generated.  It is
enough, therefore, to show that $B$ is finitely generated.  To see this,
first note that $B$ is the set of $\Z$-points of the $\Q$-algebraic
group consisting of matrices in $\SL(2g,\Q)$ which preserve the
(integral) subgroup $W < \Q^{2g}$ and the symplectic form and (for $\Gamma =
\T^W_{g,b}$) which act trivially on $H/W$.  Hence $B$ 
is arithmetic, and as such it is finitely generated (see,
e.g. \cite{Mo}, Theorem 5.57).
\end{proof}

\noindent
We now prove Theorems~\ref{thm:lagrangian_distortion} and \ref{thm:pullback_distortion}
from the introduction.

\begin{proof}[Proof of Theorem~\ref{thm:lagrangian_distortion}]
Let $H = \HH_1(\Sigma_{g,b};\Z)$ with symplectic basis
$\{a_1,\ldots,a_g,b_1, \ldots, b_g\}$ and let $W = \langle a_1, \ldots
a_g \rangle$ be the standard Lagrangian. Then the Lagrangian subgroup
$L$ is $\T^W_{g,b}$.  By Lemma~\ref{lemma:stupid_distortion_lemma}
and Proposition~\ref{proposition:relativetorellifinite}, to prove that
$\T^W_{g,b}$ is at least exponentially distorted in $\Mod_{g,b}$, it is
enough to prove that $\T^W_{g,b}$ is at least exponentially distorted in
$\Mod^W_{g,b}$.  The proof is based on the exact sequence
\begin{equation}
\label{lagrangianseq}
1 \longrightarrow \T^W_{g,b} \longrightarrow \Mod^W_{g,b} \longrightarrow A \longrightarrow 1,
\end{equation}
where $A < \Sp(2g,\Z)$ consists of symplectic matrices with a $g \times
g$ block of zeros in the lower left-hand corner.  Now, by
Remark~\ref{remark:h_image}, we have $H^W = 0$.  By Lemma~\ref{lemma:johnsonimage}, the relative Johnson homomorphism $\tau^W$ given by
Theorem~\ref{theorem:relativejohnson} is of the form $$\tau^W :
\T^W_{g,b} \longrightarrow \HH_3(H/W;\Z) \cong \wedge^3(H/W) \cong
\wedge^3 \langle b_1,\ldots,b_g \rangle.$$ Moreover, it is easy to see
that $A$ acts on $\HH_3(H/W;\Z)$ via the surjection $A \to \SL(g,\Z)$
given by projection to the lower right hand $g \times g$ block.  
For $g \geq 4$, Theorem~\ref{thm:irreducibles_agree} says that
the $A$-module $\wedge^3 \Z^g$ is nontrivial and
irreducible, so we can apply Propositions~\ref{proposition:rep_criterion} and
\ref{proposition:lower_bound_distortion} to $\T^W_{g,b} \normal
\Mod^W_{g,b}$, setting $\psi=\tau^W$, 
to conclude that $\T^W_{g,b}$ is at least exponentially
distorted in $\Mod^W_{g,b}$, as desired.
\end{proof}

\begin{proof}[Proof of Theorem~\ref{thm:pullback_distortion}]
First note that $K = \T^{W}_{g}$, where $W = \HH_1(\Sigma_{h,1};\Z)$.  
Again, by Lemma~\ref{lemma:stupid_distortion_lemma} and Proposition~\ref{proposition:relativetorellifinite}, to prove
that $\T^W_{g}$ is at least exponentially distorted in $\Mod_{g}$, it is enough to prove that $\T^W_{g}$ is at
least exponentially distorted in $\Mod^W_{g}$.  To prove this we
consider the exact sequence
\begin{equation}
\label{pullbackseq}
1 \longrightarrow \T^W_{g} \longrightarrow \Mod^W_{g} \longrightarrow A \longrightarrow 1,
\end{equation}
where $A < \Sp(2g,\Z)$ is isomorphic to $\Sp(2(g-h),\Z)$.  Now, 
by Lemma~\ref{lemma:johnsonimage}, the homomorphism $\tau^W$ 
given by Theorem~\ref{theorem:relativejohnson} is of the form
$$\tau^W : \T^W_{g} \longrightarrow \HH_3(H/W;\Z) / H^W \cong (\wedge^3 \HH_1(\Sigma_{g-h,1};\Z)) / \HH_1(\Sigma_{g-h,1};\Z).$$
Since $g-h \geq 2$, Theorem~\ref{thm:irreducibles_agree} says that the $\Sp(2(g-h),\Z)$-module
$$(\wedge^3 \HH_1(\Sigma_{g-h,1};\Z)) / \HH_1(\Sigma_{g-h,1},\Z)$$
is nontrivial and irreducible.  Applying Propositions~\ref{proposition:rep_criterion} and \ref{proposition:lower_bound_distortion}
to $\T^W_{g} \normal \Mod^W_{g}$ with $\psi=\tau^W$ then gives
that $\T^W_{g}$ is at least exponentially
distorted in $\Mod^W_{g}$, as desired.
\end{proof}

\noindent
The following theorem gives lower bounds for the distortion of
$\T^W_{g,b}$ in $\Mod_{g,b}$ for arbitrary $W$.  As the proof is 
similar to the proofs of
Theorems~\ref{thm:lagrangian_distortion} and \ref{thm:pullback_distortion}, we omit it.

\begin{theorem}
\label{thm:gentorellidistorted}
Let $g \geq 3$, $b \in \{0,1\}$, and $W \subset H$.  Assume that some element of $\Mod^W_{g,b}$ acts partially hyperbolically on
$\HH_3(H/W;\Z) / H^W$.
Then $\T^W_{g,b}$
is at least exponentially distorted in $\Mod_{g,b}$.
\end{theorem}


\noindent
Dept. of Mathematics, University of Chicago\\
5734 S. University Ave.\\
Chicago, IL 60637\\

\noindent
E-mails: {\tt broaddus@math.uchicago.edu} (Nathan Broaddus),
{\tt farb@math.uchicago.edu} (Benson Farb), {\tt andyp@math.uchicago.edu} (Andrew Putman)

\end{document}